\documentclass[11pt,a4paper]{article}
\usepackage{amsmath,amsthm,amssymb,amsfonts}%,stmaryrd}
\usepackage{hyperref}
\usepackage{enumitem}

\newcommand{\C}{{\mathbb{C}}}          % \C       = complexos
\newcommand{\F}{{\mathbb{F}}}          % \F       = field
\newcommand{\Hamil}{{\mathbb{H}}}      % \Hamil   = hamiltonianos
\newcommand{\N}{{\mathbb{N}}}          % \N       = naturais
\newcommand{\Octoni}{{\mathbb{O}}}     % \Octoni  = octoniões
\newcommand{\Proj}{{\mathbb{P}}}        % \Pro     = projectivo
\newcommand{\Z}{{\mathbb{Z}}}          % \Z       = inteiros

\newcommand{\Gdois}{{\mathrm{G}_2}}

\newcommand{\SO}{{\mathrm{SO}}}

\newcommand{\rr}{\rightarrow}
\newcommand{\lrr}{\longrightarrow}

\newcommand{\calf}{{\cal F}}             %
\newcommand{\call}{{\cal L}}             %
\newcommand{\calp}{{\cal P}}             %
\newcommand{\calx}{{\cal X}}             %

\newcommand{\na}{{\nabla}}

\newcommand{\cotg}{{\mathrm{cotg}}\,}

\newcommand{\dx}{{\mathrm{d}}}
   
\newcommand{\pr}{{\mathrm{P}}}   
\newcommand{\pfaff}{{\mathrm{Pf}}}

\newcommand{\inv}[1]{{#1}^{-1}}
\newcommand{\papa}[2]{\frac{\partial#1}{\partial#2}}

\newcommand{\vol}{{\mathrm{vol}}}

\newtheorem{teo}{Theorem}[section]

\newtheorem{prop}{Proposition}[section]

\newtheorem{defi}{Definition}[section]

\newenvironment{Rema}[1][Remark.]{\begin{trivlist}
\item[\hskip \labelsep {\bfseries #1}]}{\end{trivlist}}

\def\cyclic{\mathop{\kern0.9ex{{+}
\kern-2.2ex\raise-.28ex\hbox{\Large\hbox{$\circlearrowright$}}}}\limits}

\title{The Euler characteristic of hypersurfaces in space forms and applications to isoparametric hypersurfaces}

\author{R. Albuquerque}

\begin{document}

\maketitle

\begin{abstract}

We revisit Allendoerfer-Weil's formula for the Euler characteristic of 
embedded hypersurfaces in constant sectional curvature manifolds, first 
taking some time to re-prove it while demonstrating techniques of 
\cite{Alb2012} and then applying it to gain new understanding of 
isoparametric hypersurfaces.
\end{abstract}

\ 
\vspace*{5mm}\\
{\bf Key Words:} Euler characteristic, hypersurface, mean curvatures, space form, isoparametric hypersurface.
\vspace*{2mm}\\
{\bf MSC 2010:} Primary:  53C35,  57R20; Secondary: 53C17, 58A15

\vspace*{10mm}

\markright{\sl\hfill  R. Albuquerque \hfill}

\setcounter{section}{1}

\begin{center}
 \textbf{1. Allendoerfer-Weil's formula}
\end{center}
\setcounter{section}{1}

Let $(N,g)$ denote an oriented $C^2$ Riemannian $(n+1)$-dimensional manifold.

By a hypersurface $f:M\hookrightarrow N$ we refer to a closed orientable 
$C^2$ embedded $n$-dimensional Riemannian submanifold $M$ of $N$. The second 
fundamental form $A=\na_{\cdot}\vec{n}$, where $\vec{n}$ denotes a unit 
vector field normal to $M$, becomes a class $C^0$ endomorphism of $TM$. The 
$i$th-mean curvature $H_i$, $1\leq i\leq n$, is defined by
\begin{equation}
 H_i=\frac{1}{\binom{n}{i}} \sum_{1\leq j_1<\cdots<j_i\leq n}\lambda_{j_1}\cdots\lambda_{j_i} 
\end{equation}
where $\lambda_1,\ldots,\lambda_n$ are the eigenvalues of $A$, the so-called principal curvatures of $M$. In other words, $\binom{n}{i}H_i=S_i$ is the elementary symmetric polynomial of degree $i$ on the $\lambda_j$. We let $H_0=1$. One may define the invariant $S_i$ by the generating polynomial $\det(1+xA)=\sum_{i=0}^n S_ix^i$.

Now let us suppose $n$ is even, say $n=2l\in\N$, and that $N=N(c)$ is an 
$(n+1)$-dimensional manifold of constant sectional curvature $c$.

In this article we start, right from this first section, by revisiting the 
C.~B.~Allendoerfer and A.~Weil's formula \cite[formula 19]{Allendoerfer} 
for the Euler or Euler-Poincar\'e characteristic of an embedded hypersurface 
$M$ in $N$. According to Allendoerfer, the next result is proved in 
\cite{Allen-Weil}:
\begin{equation} \label{CaractEuler_intro}
 \chi(M) \:=\: \frac{1}{2^l\pi^l}\sum_{p=0}^{l}
 (2l-2p-1)!!\,(2p-1)!!\,c^{l-p}\int_M S_{2p}\,\vol .
\end{equation}
We give two new proofs of the invariance of the previous quantity under 
$C^2$ deformation of $M$. One proof, still in Section 1, is based on a 
formula of R.~Reilly in \cite{Rei} for the derivative $\frac{\dx}{\dx 
t}\int_{M_t}S_i\,\vol_t$, where of course the $S_i$ vary with $M_t$. Another 
proof, in Sections 2 and 3, is found as an application both of the theory of 
Euler-Lagrange systems from \cite{BGG} and the fundamental exterior 
differential system of Riemannian geometry, which was found by the author in 
\cite{Alb2012} and which, in particular, proves Reilly's formula. In Section 
4 we confirm with the Theorem of Gauss-Bonnet-Chern that the referred 
invariant is indeed $\chi(M)$.

In Sections 5 and 6 we explore some of the further results stated in 
\cite{Allendoerfer}, which are all quite important to recall.

Finally, in Section 7, we establish new formulas between the 
volume, the Euler characteristic and the constants $\lambda_i$ of certain 
isoparametric hypersurfaces. These are indeed original applications of 
Allendoerfer-Weil's formula.

We proceed immediately with the announced research plan.
\begin{defi}
 Let $\pr$ denote the function on $M$, dependent on the 2nd fundamental form,
 \begin{equation} \label{thedifferentialinvariant}
 \begin{split}
 \pr &  = \pr(\lambda_1,\ldots,\lambda_n)  \\
 & = (2l-1)!!\,c^{l}+(2l-3)!!\,c^{l-1}\,S_2+\cdots +   \\
 &\hspace{34mm} (2l-2p-1)!!\,(2p-1)!!\,c^{l-p}\,S_{2p}
+\cdots+(2l-1)!!\,S_{2l} \\
 & = (2l-1)!! \sum_{p=0}^{l}\binom{l}{p}c^{l-p}H_{2p} 
 \end{split}
 \end{equation}
where $(2p-1)!!=(2p-1)(2p-3)\cdots5\cdot3\cdot1$ and $(-1)!!=1$.
\end{defi}

An alternative expression for $\pr$ in \eqref{thedifferentialinvariant} uses 
the trivial identity $(2p)!/p!=2^p(2p-1)!!$.

\begin{teo}        \label{TheoremInvariantbyDeformation}
Let $f_t:M\rr N(c)$ denote a $C^2$ deformation of a hypersurface $f_0:M\rr 
N(c)$ with empty boundary. Then, up to a scalar multiple, $\pr$ is the only 
constant coefficients linear combination on the $S_i$ which yields
 \begin{equation}
   \frac{\dx}{\dx t}\int_{M_t}\pr\,\vol_t = 0 .
 \end{equation}
\end{teo}
\begin{proof}
Immediate application of Reilly's formula (a) in \cite[Theorem B]{Rei}, 
valid for all $t$:
 \begin{align}   \label{formulaReilly}
  \frac{\dx}{\dx t}\int_{M_t}S_i\,\vol_t &=
     \int_{M_t}g(X,\vec{n})((i+1)S_{i+1}-c(n-i+1)S_{i-1})\,\vol_t ,   
 \end{align}
for any $0\leq i\leq n$, where $X=\papa{f_t}{t}$ is the deformation vector field.
\end{proof}

We thus have the integral invariant
\begin{equation} \label{CaractEuler}
 \chi(M) = \frac{1}{2^l\pi^l}\int_M \pr\,\vol = \frac{1}{2^l\pi^l} 
     \sum_{p=0}^{l} (2l-2p-1)!!\,(2p-1)!!\,c^{l-p}\int_M S_{2p}\,\vol .
\end{equation}
We shall see below that the expression on the right hand side is indeed 
the Euler-Poincar\'e characteristic $\chi(M)$ of $M$. In other words, that 
$\frac{1}{2^l\pi^l}\pr\,\vol_M$ is the $n$-form found by the Theorem of 
Gauss-Bonnet-Chern. Thus we should not have to show that \eqref{CaractEuler} 
is a constant under deformations, leaving aside the uniqueness statement.

Formula \eqref{CaractEuler} is due to C.~Allendoerfer and A.~Weil, 
according to the first in \cite{Allendoerfer} who points the reader to 
\cite{Allen-Weil}. But indeed we can only find it explicit in 
\cite[formula 19]{Allendoerfer}. The result has also appeared in other 
references, deduced through different means. Namely, a proof is given by 
G.~Solanes in \cite{Solanes} using methods of integration over the
intersections of $M$ with \textit{certain} {planes}. Interesting enough, 
Allendoerfer-Weil's formula is shown in the celebrated book of L.~Santal\'o 
\cite[p.303]{Santalo1} and the formula of R.~Reilly is deduced by Santal\'o 
\cite{Santalo2} for parallel deformations of the given hypersurface.

\vspace{2mm}
\begin{center}
 \textbf{2. Variational calculus with the tangent sphere bundle}
\end{center}
\setcounter{section}{2}

The theory of a fundamental exterior differential system ``$\theta,\alpha_0,\ldots,\alpha_n$'' of Riemannian geometry was first presented in \cite{Alb2012}. Applications in various fields soon emerged, some also related with the integral of mean curvatures of hypersurfaces, cf. \cite{Alb2012,Alb2015a,Alb2018}.

Again we let $(N,g)$ denote an oriented Riemannian $(n+1)$-dimensional manifold of class $C^2$. We consider the unit tangent sphere bundle $\pi:SN\lrr N$ with its natural Sasaki metric and induced $\SO(n)$ structure. The metric gives the horizontal tangent subbundle and the geodesic flow vector field; its dual is the most fundamental contact 1-form on $SN$, denoted by $\theta$.

We recall the mirror endomorphism $B:TTN\rr TTN$, which sends a horizontal lift of a vector field on $N$ to the respective vertical lift and sends verticals to 0. $B$ is indeed an endomorphism and an important object in the definition of the $\alpha_i\in\Omega^n_{SN}$, $i=0,\ldots,n$. Letting $\alpha_n$ be the volume-form of the fibers, we recall
\begin{equation}
  \alpha_i \:=\: \frac{1}{i!(n-i)!}\,\alpha_n\circ(B^{n-i}\wedge1^i) ,
\end{equation}
with $\circ$ denoting a natural alternating operator. It is also useful to define $\alpha_{-1}=\alpha_{n+1}=0$ and to recall the identity $\dx\theta\wedge\alpha_i=0$, for all $i=0,\ldots,n$.

Now, let $f:M\rr N$ denote an oriented hypersurface. On $M$ we have the 
canonical lift $\hat{f}:M\hookrightarrow SN$, defined by 
$\hat{f}(x)=\vec{n}_{f(x)}$. We have the following formulas, cf. 
\cite[Proposition 3.2]{Alb2012}:
\begin{equation} \label{fundamentalformulapullback}
\begin{split}
 &\qquad\hat{f}^*\theta=0, \\
 &\hat{f}^*{\alpha_i}= S_i\,\vol_M .
\end{split}
\end{equation}

As in \cite{Alb2012,Alb2018,BGG}, we now consider a Lagrangian 
$\Lambda\in\Omega_{SN}^n$. This is just any given $n$-form on $SN$. Then we 
may develop the variations of the following functional on $M$
\begin{equation}\label{calF_Lambda}
 \calf_\Lambda(M_t) := \int_{M}\hat{f_t}^*\Lambda .
 \end{equation}

\begin{Rema}
Let $X$ be any vector field on the base space $N$ and let $\{\phi_t:N\rr N\}$ denote its flow. Then the flow of the horizontal lift $\pi^*X$ of $X$ on the unit tangent sphere bundle is rather difficult to find. A rather natural object to consider is the complete lift or extension of $X$ onto $TN$. Nevertheless we shall observe some results where horizontal lifts are just helpful enough.
\end{Rema}

The vector field $X_t=\papa{f_t}{t}$ is related with the derivative of \eqref{calF_Lambda}. We shall look for variations of $M$ with the boundary fixed. Hence $X$ vanishes on $\partial M$. The derivative requires us to find the Poincar\'e-Cartan form $\theta\wedge\Psi_1$ associated to $\dx\Lambda$, this is, the forms $\theta\wedge\Psi_1\in\Omega^{n+1}$ and $\beta_1\in\Omega^{n-1}$ such that $\dx\Lambda=\theta\wedge\Psi_1+\dx(\theta\wedge\beta_1)$ --- a decomposition assured to exist by a Theorem of Lepage, \cite{BGG}. We then find
\begin{equation}  \label{thevariationalderivative}
\begin{split}
 \lefteqn{\boldsymbol{\delta}_{_{X_t}}\calf_\Lambda(M_t)= \frac{\dx}{\dx t}\int_{M}\hat{f_t}^*(\Lambda-\theta\wedge\beta_1)  
   =\int_M\hat{f_t}^*\call_{\calx_t}(\Lambda-\theta\wedge\beta_1) } \hspace*{45mm} \\ 
& = \int_M\hat{f_t}^*\bigl(\dx({\calx_t}\lrcorner(\Lambda-\theta\wedge\beta_1)) + {\calx_t}\lrcorner(\theta\wedge\Psi_1)\bigr) \\ 
& =\int_{\partial M}\hat{f_t}^*{\calx_t}\lrcorner(\Lambda-\theta\wedge\beta_1)+ \int_M\hat{f_t}^*(\theta(\calx_t)\,\Psi_1) .
 \end{split}
 \end{equation}
Notice we must take $\calx_t=\papa{\hat{f_t}}{t}$ for the Lie derivative. But the horizontal part $\calx_t^h$ is the horizontal lift of $X_t$. Easy enough, $\pi_*(\calx_t) =\papa{ }{t}\pi\circ\hat{f_t}
 =\papa{f_t}{t}=X_t$. Since, by definition, 
$\theta_u(Z)=g_{\pi(u)}(\pi_*Z,u)$ for all $u\in SN$ and $Z\in TSN$, we 
find
 \[  \hat{f_t}^*(\theta(\calx_t))=(\theta(\calx_t))_{\hat{f_t}} 
=g(X_t,\vec{n}).    \]
\begin{prop} \label{stationarysubmanifold_iff_condition}
 For a hypersurface $M\subset N$ to be a stationary submanifold of the functional $\calf_\Lambda$, under all variations with fixed boundary $\partial M$ and fixed unit normal to the boundary, it is necessary and sufficient that $\hat{f}^*\Psi_1=0$.
\end{prop}
Notice the hypothesis ${\calx_t}_{\mid\partial M}=0$. These are deformations which leave $\partial M$ strongly fixed in the sense that both $X_t$ and its gradient vanish on $\partial M$ (as in \cite[p.467]{Rei}). Of course, if the boundary is empty, then there is no restriction on the deformation.
 
\begin{Rema}
 On the assumptions of Proposition \ref{stationarysubmanifold_iff_condition}, we may compute the second derivative:
\begin{equation}
  \boldsymbol{\delta}^2_{_{X_t}}\calf_\Lambda(M_t)
   = \int_M(\frac{\dx}{\dx t} g(X_t,\vec{n}))\hat{f_t}^*\Psi_1+g(X_t,\vec{n})\hat{f_t}^*\call_{\calx_t}\Psi_1 .
 \end{equation}
If $\theta\wedge\Psi_2$ is the new Poincar\'e-Cartan form, $\dx\Psi_1\equiv\theta\wedge\Psi_2+\dx(\theta\wedge\beta_2)$, then on the stationary $M=M_0$ we obtain
\begin{equation} \label{2nd_deri_calF_sta}
  \boldsymbol{\delta}^2_{_{X}}\calf_\Lambda(M)\,_{|t=0}
   = \int_Mg(X,\vec{n})\hat{f}^*\bigl(\dx(\calx\lrcorner\Psi_1)+\calx\lrcorner\dx(\theta\wedge\beta_2)\bigr)+(g(X,\vec{n}))^2\hat{f}^*\Psi_2 .
 \end{equation}
\end{Rema}

\vspace{2mm}
\begin{center}
 \textbf{3. Alternative proof of Theorem \ref{TheoremInvariantbyDeformation}}
\end{center}
\setcounter{section}{3}

Let us return to the case of an oriented Riemannian $(n+1)$-dimensional 
manifold of constant sectional curvature $N=N(c)$ and let us study the 
functional $\calf_{\alpha_i}$. We have the following \textit{magic} formula 
from \cite[formula 2.26]{Alb2012}:
\begin{equation}\label{derivadasalpha_iemCSecC}
   \dx\alpha_i=
     \theta\wedge\bigl((i+1)\alpha_{i+1}-c(n-i+1)\alpha_{i-1}\bigr) .
\end{equation}
These are already in the Poincar\'e-Cartan form. Combining with 
(\ref{fundamentalformulapullback},\ref{thevariationalderivative}), yields 
immediately formula (a) of R.~Reilly in \cite[Theorem B]{Rei} which we have
used in \eqref{formulaReilly}.

Furthermore, a straightforward application of \eqref{2nd_deri_calF_sta} 
improves on the second part of that theorem, formula (b). Indeed, from 
above we have $\dx\alpha_i=\theta\wedge\Psi_1$ and thus 
$\dx\Psi_1=\theta\wedge\Psi_2$ where
\begin{equation}
\Psi_2=(i+1)(i+2)\alpha_{i+2}
-c(i(n-i+1)+(i+1)(n-i))\alpha_i+c^2(n-i+1)(n-i+2)\alpha_{i-2}.
\end{equation}
Hence the following result.
\begin{prop}
Under all variations with fixed boundary $\partial M$ and fixed unit normal 
to the boundary, at a stationary submanifold of $\calf_{\alpha_i}$ we have
\begin{equation}  \label{thevariationalderivativeReillyPartB}
 \boldsymbol{\delta}^2_{_{X}}\calf_{\alpha_i}(M_t)\,_{|t=0} =\left. 
\frac{\dx^2}{\dx t^2}\right\vert_{t=0}\,\int_{M}S_i\,\vol_t 
=\int_Mg(X,\vec{n})\hat{f}^*\dx(\calx\lrcorner\Psi_1)+ 
 g(X,\vec{n})^2\hat{f}^*\Psi_2 .
\end{equation}
\end{prop}

Now any non-trivial Lagrangian $\Lambda$ which satisfies $\dx\Lambda=0$ 
gives rise to a \textit{conservation law}: here, an integral invariant of 
hypersurfaces under $C^2$ deformations. Thus we consider a form
\begin{equation}
 \Lambda=b_0\alpha_0+\cdots+b_n\alpha_n
\end{equation}
with constant coefficients $b_j$. Such a form is an invariant of the geodesic flow vector field $S=\theta^\sharp$.
\begin{prop}
 If $\dx\Lambda=\theta\wedge\Lambda_1$, then $\Lambda_1=\call_S\Lambda$.
\end{prop}

\begin{prop}  \label{Prop_dLambda_vanishingcases}
\,$\dx\Lambda=0$ if and only if
 \begin{itemize}
   \item in case $c=0$, we have $b_0=b_1=\ldots=b_{n-1}=0$
   \item in case $c\neq0$ and $n$ is odd, we have $b_0=b_1=\ldots=b_{n-1}=b_n=0$,
   \item in case $c\neq0$ and $n$ is even, we have $b_1=b_3=b_5=\ldots=b_{n-1}=0$ and
     \begin{equation} \label{fuel2}
  b_{2p}=\frac{1\cdot3\cdot5\cdots(2p-1)b_0}{c^p(n-1)(n-3)
   \cdots(n-2p+1)},\ \ \forall p=1,\ldots,\tfrac{n}{2}.
   \end{equation}
\end{itemize}
\end{prop}
\begin{proof}
It is clear that
\begin{equation}
 \dx \Lambda=\theta\wedge\sum_{j=0}^n(jb_{j-1}-c(n-j)b_{j+1})\alpha_j .
\end{equation}
 The $\alpha_0,\ldots,\alpha_n$ are linearly independent, so from the cases 
 $j=0,\,n$ we obtain $cb_1=0,\,b_{n-1}=0$, respectively. We may immediately 
 assume $c\neq0$. Then $b_1=0$; and then $b_3=0$; and so on and so forth. 
 If also $n$ is odd, then it follows $b_n=0$. Moreover, in the descending 
 order, $b_{n-1}=0$ implies all the $b_{\mathrm{even}}=0$. Finally, if $n$ 
 is even, then the last remaining condition must hold: $b_{\mathrm{odd}}=0$ 
 and an induction on $jb_{j-1}-c(n-j)b_{j+1}=0$ yields the formula.
\end{proof}
In case $c=0$, we may have $b_n\neq0$ and thus find a variational trivial functional. $S_n$ is the so-called Gauss-Kronecker curvature.

Let $n=2l$ and let us take $b_0=c^l(2l-1)!!$. Then
\begin{equation}  \label{thedifferentialinvariant_withlambda}
 \hat{f_t}^*\Lambda =\sum_{i=0}^{n}b_iS_i\,\vol_M=\sum_{p=0}^{l}
 (2l-2p-1)!!(2p-1)!!c^{l-p}S_{2p}\,\vol_M
\end{equation}
yields the invariant associated to $\pr$ as described in Theorem \ref{TheoremInvariantbyDeformation}.

We remark P.~Gilkey~et~al.~have shown that there is only one natural differential invariant, up to a multiple, which is defined for all manifolds and metrics; namely, the invariant which yields Euler characteristic, cf. \cite{NavNav}. However, since we have a restricted class of submanifolds of a restricted class of manifolds, we could not immediately rule out the existence of some other differential invariants in this context and, in particular, one dependent on the metric.

\vspace{2mm}
\begin{center}
 \textbf{4. $\chi$ is the Euler characteristic}
\end{center}
\setcounter{section}{4}

To the best of our knowledge, the generalized Theorem of Gauss-Bonnet-Chern 
is first surveyed in its most comprehensive and insightful form in 
\cite{MilnorStasheff}, and such is the form we wish to follow here. The 
theorem establishes the identity $\chi(M)= \int_M\pfaff(\Omega/2\pi)$, where 
$\Omega=[\Omega_{ij}]_{i,j=1,\ldots,n}$ is the curvature 2-form matrix and 
$\pfaff(\Omega)$ is the Pfaffian of $\Omega$, a closed $n$-form and a 
non-trivial object in even dimensions in general.

On a constant sectional curvature ambient manifold $N(c)$, for a hypersurface $f:M\rr N(c)$ and a positively oriented orthonormal frame $e_1,\ldots,e_n$ of principal curvatures, we have by Gauss-Codazzi equations, cf.\cite{Alb2012}:
\begin{equation}
 R^M_{pkij}=(c+\lambda_i\lambda_j)(\delta_{pi}\delta_{kj}-
               \delta_{pj}\delta_{ki}) .
\end{equation}
Hence the matrix of 2-forms corresponds to $\Omega_{ij}=(c+\lambda_i\lambda_j)e^{ij}$. The Pfaffian $n$-form is the Pfaffian of the following skew-symmetric matrix
\begin{equation}\label{GaussCodazzimatrix}
 \left[\begin{array}{cccccc}
       0 & (c+\lambda_1\lambda_2)e^{12} & (c+\lambda_1\lambda_3)e^{13} &  & & (c+\lambda_1\lambda_n)e^{1n}  \\
     (c+\lambda_1\lambda_2)e^{21} & 0 & (c+\lambda_2\lambda_3)e^{23} & \cdots & & (c+\lambda_2\lambda_n)e^{2n}  \\         
     (c+\lambda_1\lambda_3)e^{31} &  (c+\lambda_2\lambda_3)e^{32} & 0 & & & (c+\lambda_3\lambda_n)e^{3n}  \\ 
     & \cdots & & \ddots & &  \\
     (c+\lambda_1\lambda_n)e^{n1} & (c+\lambda_2\lambda_n)e^{n2} &  &  & &0  \end{array}\right] .
\end{equation}
Let us denote such matrix by $\Omega(c;\lambda_1,\ldots,\lambda_n)$.

For $n=2$, clearly $\pfaff(\Omega(c;\lambda_1,\lambda_2))=(c+\lambda_1\lambda_2)e^{12}$, proving \eqref{CaractEuler}. As well as with $n=4$, where a well-known formula is applied
\begin{equation}
\begin{split}
 & \pfaff(\Omega(c;\lambda_1,\lambda_2,\lambda_3,\lambda_4)) =\Omega_{12}\wedge\Omega_{34}-\Omega_{13}\wedge\Omega_{24} 
 +\Omega_{23}\wedge\Omega_{14} \\
   & \qquad =\bigl(3c^2+c(\lambda_1\lambda_2+\lambda_3\lambda_4+\lambda_1\lambda_3+\lambda_2\lambda_4+\lambda_2\lambda_3+\lambda_1\lambda_4)
   +3\lambda_1\lambda_2\lambda_3\lambda_4\bigr)e^{1234} .
\end{split}
\end{equation}

Now, we recall a `Laplace rule' for the Pfaffian of an $n\times n$ skew-symmetric matrix $X$, cf. \cite[p.27]{DSerre}:
\begin{equation} \label{Pfaffianbyrecurrence}
 \pfaff(X)=\sum_{j=2}^n(-1)^jx_{1j}\pfaff(X^{1j})
\end{equation}
where $X^{ij}$ denotes the skew-symmetric matrix with the $i$th and $j$th rows and columns removed. Since $\Lambda^{\mathrm{even}}T^*M$ is commutative, the formula is valid for the curvature form, i.e. giving the preferred square-root of the determinant of $\Omega$.

For $n=2l$, it follows that $\pfaff(\Omega(c;0,\ldots,0))=(n-1)!!\,c^l\,\vol$. Therefore,
\begin{equation}
 \pfaff(\Omega(0;\lambda_1,\ldots,\lambda_n)) = (n-1)!!\,\lambda_1\cdots\lambda_n\,\vol .
\end{equation}
The next result finally establishes Allendoerfer-Weyl's formula 
\eqref{CaractEuler}.
\begin{prop}
 We have that $\pfaff(\Omega)=\pr\,\vol_M$, where $\pr$ is given in \eqref{thedifferentialinvariant}.
\end{prop}
\begin{proof}
 Let $A_{l,p}=(2l-2p-1)!!\,(2p-1)!!$. Then we have $A_{l,0}=A_{l,l}=(2l-1)!!$ and clearly $A_{l,p}=(2p-1)A_{l-1,p-1}=(2l-2p-1)A_{l-1,p}$. Now we shall prove the result by induction. Since it is true for $n=2$, we next assume $l>1$ and apply \eqref{Pfaffianbyrecurrence}. With a formal quotient by $\vol$,
 \begin{equation*}
 \begin{aligned}
 \lefteqn{ \frac{1}{\vol}\pfaff(\Omega) =\frac{1}{\vol}\sum_{j=2}^{2l}(-1)^j(c+\lambda_1\lambda_j)e^{1j}\wedge\pfaff(\Omega^{1j}) }\\
  &= \sum_{j=2}^{2l}(c+\lambda_1\lambda_j)\pr(\lambda_2,\ldots,\lambda_{j-1},\lambda_{j+1},\ldots,\lambda_{2l}) \\
  & = \sum_{j=2}^{2l}\sum_{p=0}^{l-1}(c+\lambda_1\lambda_j)c^{l-1-p}A_{l-1,p}\,S_{2p}(\lambda_2,\ldots,\lambda_{j-1},\lambda_{j+1},\ldots,\lambda_{2l})   \\
  & = \sum_{j=2}^{2l}\sum_{p=0}^{l-1}c^{l-p}A_{l-1,p}\,S_{2p}(\lambda_2,\ldots,\widehat{\lambda_{j}},\ldots,\lambda_{2l}) + \sum_{j=2}^{2l}\sum_{q=1}^{l}c^{l-q}A_{l-1,q-1}\,\lambda_1\lambda_jS_{2q-2}(\lambda_2,\ldots,\widehat{\lambda_{j}},\ldots,\lambda_{2l}) \\
  & = (2l-1)!!(c^l+\lambda_1\cdots\lambda_{2l}) + \sum_{q=1}^{l-1}c^{l-q} \sum_{j=2}^{2l}\biggl(A_{l-1,q}\,S_{2q}(\lambda_2,\ldots,\widehat{\lambda_{j}},\ldots,\lambda_{2l}) \biggr. \\
  &\hspace{65mm} \biggl. +A_{l-1,q-1}\lambda_1\lambda_jS_{2q-2}(\lambda_2,\ldots,\widehat{\lambda_{j}},\ldots,\lambda_{2l})\biggr)\\
  & = A_{l,0}c^l+A_{l,l}S_{2l} + \sum_{q=1}^{l-1}c^{l-q}A_{l,q}
  \sum_{j=2}^{2l}\biggl(\frac{S_{2q}(\lambda_2,\ldots,\widehat{\lambda_{j}},\ldots,\lambda_{2l})}{2l-2q-1}+\frac{\lambda_1\lambda_jS_{2q-2}(\lambda_2,\ldots,\widehat{\lambda_{j}},\ldots,\lambda_{2l})}{2q-1} \biggr).
 \end{aligned}
 \end{equation*}
The result follows once we prove, for all $1\leq q\leq l-1$,
\begin{equation*}
 \sum_{j=2}^{2l}\biggl(\frac{S_{2q}(\lambda_2,\ldots,\widehat{\lambda_{j}},\ldots,\lambda_{2l})}{2l-2q-1}+\frac{\lambda_1\lambda_jS_{2q-2}(\lambda_2,\ldots,\widehat{\lambda_{j}},\ldots,\lambda_{2l})}{2q-1} \biggr)
 =S_{2q}(\lambda_1,\ldots,\lambda_{2l}) .\qquad\quad(\star)
 \end{equation*}
Firstly, arguing by symmetry and with each monomial, we find
\begin{equation*}
\begin{aligned}
 & S_{2q}(\lambda_3,\lambda_4,\ldots,\lambda_{2l})+S_{2q}(\lambda_2,\lambda_4,\ldots,\lambda_{2l})+\cdots+S_{2q}(\lambda_2,\lambda_3,\ldots,\lambda_{2l-2},\lambda_{2l})+S_{2q}(\lambda_2,\lambda_3,\ldots,\lambda_{2l-1}) \\  &\qquad=\ (2l-2q-1) S_{2q}(\lambda_2,\lambda_3,\lambda_4,\ldots,\lambda_{2l-1},\lambda_{2l}).
\end{aligned}
\end{equation*}
Secondly, another careful inspection gives
\begin{equation*}
\begin{aligned}
 & \lambda_2S_{2q-2}(\lambda_3,\lambda_4,\ldots,\lambda_{2l})+\lambda_3S_{2q-2}(\lambda_2,\lambda_4,\ldots,\lambda_{2l})+\cdots+\lambda_{2l}S_{2q-2}(\lambda_2,\lambda_3,\ldots,\lambda_{2l-1}) \\  &\qquad=\ (2q-1) S_{2q-1}(\lambda_2,\lambda_3,\lambda_4,\ldots,\lambda_{2l-1},\lambda_{2l}).
\end{aligned}
\end{equation*}
Hence the left hand side of ($\star$) becomes
\begin{equation*}
 =\: S_{2q}(\lambda_2,\ldots,\lambda_{2l})+\lambda_1S_{2q-1}(\lambda_2,\ldots,\lambda_{2l}) \:=\: S_{2q}(\lambda_1,\ldots,\lambda_{2l}) 
\end{equation*}
as we wished.
\end{proof}

Let us see a simple formula. For two principal curvatures $\lambda,\mu$ of multiplicities respectively $1,n-1$ we have
\begin{equation}   \label{specialcurvature}
\begin{split}
 \pfaff(\Omega(c;\lambda,\mu,\ldots,\mu)) &=  (c+\lambda\mu)\sum_{j=2}^{2l}(-1)^je^{1j}\wedge\pfaff(\Omega(c+\mu^2;0,\ldots,0)) \\  &= (2l-1)!!
 (c+\lambda\mu)(c+\mu^2)^{l-1}\,\vol  .
 \end{split}
\end{equation}
This can be checked either through \eqref{GaussCodazzimatrix} or \eqref{thedifferentialinvariant}.

Example: We consider the canonical immersion of hypersurfaces $S^{p,q}_r:=S^p_r\times S^q_{\sqrt{1-r^2}}\hookrightarrow S^{n+1}_1$ for any integers $0\leq p,q \leq n=2l=p+q$ and $r\in]0,1[$. We shall need the formula for the volume $\omega_m$ of an $m$-sphere, $m\in\N$, of radius 1, which is $\omega_m:=2\pi^{(m+1)/2}/\Gamma((m+1)/2)$, where $\Gamma$ is the Gamma function.

It is well known that $\chi(S^{p,q}_r)=0$, for $p,q$ odd; it is equal to $2$, for $p=0$ or $q=0$; equal to $4$, for $p,q\neq0$ even. Consulting \cite[Example 2]{AliasBrasilPerdomo}, cf. \cite{Alb2018}, we find the principal curvatures:
\begin{equation}\label{principalcurvaturesCliffordSpheres}
 \lambda_1=\cdots=\lambda_p=-\frac{\sqrt{1-r^2}}{r},\quad
 \lambda_{p+1}=\cdots=\lambda_n=\frac{r}{\sqrt{1-r^2}}.
\end{equation}
The two values, say $\lambda,\mu$, respectively, satisfy $1+\lambda\mu=0$. By \eqref{specialcurvature} it follows that ${\chi}(S^{n-1,1}_r)=0$ as expected. The following results are much easier to find directly from the Pfaffian of \eqref{GaussCodazzimatrix}. In fact, since $c+\lambda\mu=0$, the multiplicative nature of $\chi$ on $S^{p,q}_r$ is immediately seen through the matrix computation. Yet the following is still interesting to see using $\pr\,\vol$, that is, by \eqref{specialcurvature}:
\begin{equation}\label{characteristicEPspheres}
 \chi(S^{n}_r)= \frac{1}{2^l\pi^l}\int\pfaff(\Omega)= \frac{(2l-1)!!}{2^l\pi^l}\int_{S^{n}_r}\bigl(1+\frac{1-r^2}{r^2}\bigr)^l\vol_r = \frac{(2l-1)!!}{2^l\pi^l}\vol(S^n_1)=2. 
\end{equation}
In the same way, one finds
\begin{equation}
\chi(S^{2,2}))=4=(\chi(S^2))^2=\chi(S^{2,4})).
\end{equation}

By \eqref{characteristicEPspheres} there remains no doubt that dividing the invariant $\int_M\pr\,\vol$ by $(2\pi)^l$ yields always an integer in the equivalence class of the $2l$-sphere.

\vspace{2mm}
\begin{center}
 \textbf{5. Chern's proof of Gauss-Bonnet and a theorem of H.~Hopf}
\end{center}
\setcounter{section}{5}

Let us try to follow S.S.~Chern's proof in \cite{Chern} of the celebrated result bearing his name, in the case of $N=N(c)$, of course of even dimension. Thus now $n$ is odd.

A striking passage recurs to Hopf's index Theorem and therefore to the choice of a unit vector field $X$ with isolated zeros. Indeed, the unit sphere bundle plays a fundamental role in \cite{Chern}. Since $\pi$ is a Riemannian submersion, the integral of any form $\tilde{\Omega}$ over $N$ is the same as the integral of the pullback form over the $(n+1)$-submanifold or cycle $V=X(N\backslash\{\mbox{singular points}\})$ of $SN$. Moreover, $-\partial V$  equals the $n$-cycle $S^n_1\chi(N)$, where $S^n_1$ represents the standard fiber or fundamental class, cf. \cite{Chern,MilnorStasheff}. In our understanding, it is only what later became known as Thom duality, brought in with Stokes identity, which sustains Chern's swift and ingenious breakthrough.

The proof requires an \textit{intrinsic} or metric-invariant global $n$-form $\Pi$, which we search here for within the exterior differential system of the $\{\alpha_i\}$ and which satisfies both $\dx\Pi$ equal to $\pi^*\tilde{\Omega}$ when restricted to $V$, for some $\tilde{\Omega}$, and that the restriction of $\Pi$ on the fibers coincides with that of $b_n\alpha_n$, with $b_n=-1/\omega_n$ and $\omega_n=\vol(S^n_1)$. We recall $\alpha_n$ is the volume form on the fibers. Hence the desired identities
\begin{equation}
 \int_{N}\tilde{\Omega}=\int_{\partial V}\Pi=\chi(N) .
\end{equation}
The only invariant $n+1$-form we have as a pullback is $\theta\wedge\alpha_0=\pi^*\vol_N$, so we are bound to find $c_0$ from $\dx\Pi=\theta\wedge\sum c_j\alpha_j$ where $\Pi=\sum b_j\alpha_j$. Proceeding as in Proposition \ref{Prop_dLambda_vanishingcases}, we clearly have the equations $jb_{j-1}-c(n-j)b_{j+1}=c_j$, of which there are plenty of solutions in the setting and in case $c\neq0$. A simple one is $c_j=0$, $j>0$. Thus
\begin{equation}
 b_{2j}=0,\qquad\quad b_{2j+1}=-\frac{(n-2j-2)!!(2j)!!c^{\frac{n-2j-1}{2}}}{(n-1)!!\omega_n}.
\end{equation}
Hence
\begin{equation} \label{bum_czero}
 b_1=-\frac{(n-2)!!c^{\frac{n-1}{2}}}{(n-1)!!\omega_n},\qquad \quad
 c_0=\frac{n!!c^{\frac{n+1}{2}}}{(n-1)!!\omega_n}.
\end{equation}

With $c>0$ we find, for any naturals $k$, $n=2k+1$, and for the sphere $N=S^{2k+2}_{r}$, $r=\frac{1}{\sqrt{c}}$, that $2=\chi(N)=\int_N c_0\theta\wedge\alpha_0=c_0r^{2k+2}\omega_{2k+2}$. Hence the trivial identity $\frac{\omega_{2k+2}}{\omega_{2k+1}}=2\frac{(2k)!!}{(2k+1)!!}$. 

Returning to \eqref{bum_czero}, we may write $c_0\omega_{2k+2}=2c^{k+1}$.
In other words, Gauss-Bonnet-Chern's formula reads $2c^{\frac{n+1}{2}}\vol(N)=\chi(N)\omega_{n+1}$ for $c\neq0$. Of course $N$ must be of finite volume.

In particular we recover a classical theorem of H.~Hopf, as it is well known, with important implications to hyperbolic geometry. Let $\cal N$ be a hyperbolic space of even dimension $m$, constant sectional curvature $-1$ and  finite volume (e.g. the quotient of hyperbolic space form by a torsion-free discrete group of finite order). Then its volume is the topological invariant $\vol({\cal N})=(-1)^{m/2}\frac{\omega_{m}}{2}\chi({\cal N})$, a cornerstone in the arithmetic theory of hyperbolic geometry, cf. \cite{Belo,KellerhalsZehrt}.

\vspace{2mm}
\begin{center}
 \textbf{6. Other formulas from Allendoerfer's article}
\end{center}
\setcounter{section}{6}

We continue with the previous setting. Let $n$ be odd and $(N,g)$ denote an oriented $C^2$ Riemannian $(n+1)$-dimensional manifold of constant sectional curvature $c$.

Let $Q$ be a compact $(n+1)$-submanifold with $C^2$ boundary $\partial Q$. We assume the existence of a $C^1$ vector field $X$ on $Q$ such that:
\begin{itemize}
 \item $X$ has only isolated singularities and these lie in the interior of $Q$;
 \item $X$ coincides with the outward unit-normal $\vec{n}$ over $\partial 
Q$.
\end{itemize}

We may conjecture this is true for all $Q$. Certainly, an open, relatively compact and one-point contractible subset $Q\subset N$ with $C^2$ boundary admits such a vector field.

The questions raised in \cite{Allendoerfer} regarding a new formula of 
Steiner type, i.e. one which looks for the enclosed volume and 
area of hypersurfaces enlarged in geodesic outward directions, also 
refer to the case when $n$ is odd and the submanifold $Q$ is closed and 
bounding, of class $C^3$. Moreover, a complete metric on $N$ is assumed.

In particular for $\chi(Q)$, \cite[formula 20]{Allendoerfer} \textit{must 
be} the same as the one in the next theorem, for which we find a 
quick proof using the already seen argument from \cite{Chern}. Again the 
result may also be seen in \cite{Solanes}, with a complete proof.
\begin{teo} \label{Allendoerfer_odd}
Let $n=2k+1$ and $c\neq0$. In the above conditions, we have:
\begin{equation}  \label{Allendoerfer_odd_formula}
 c^{k+1}\vol(Q)+\frac{1}{n!!}\sum_{j=0}^k(2j)!!(2k-2j-1)!!\,c^{k-j} 
 \int_{\partial Q} S_{2j+1}\,\vol_{\partial Q}= 
 \frac{\chi(Q)}{2}\omega_{2k+2}.
\end{equation}
\end{teo}
\begin{proof}
 First we notice that, as a cycle in $SN$, the subspace $V=X(Q\backslash\{\mbox{singular points}\})$ has as boundary the $n$-cycle
 \[ \partial V=-\chi(Q)S^n_1\cup X(\partial Q) . \]
 Again we may find over the tangent sphere bundle an invariant $n$-form $\Pi$ such that $\dx\Pi=\theta\wedge\alpha_0$. We have seen the solution in the previous Section. Since $\theta\wedge\alpha_0=\pi^*\vol_N$, then $X^*\dx\Pi=\vol_Q$. Applying Stokes identity and \eqref{fundamentalformulapullback} the result follows.
 \end{proof}
Clearly, the result is coherent with the formulas from the last Section 
where $Q=N$. Here we are simply not chiefly focused on $\chi$. We remark 
case $n=1$ is due to W.~Blaschke.

Example 1: It is interesting to check the case $n=1$ with a sphere $N=S^2_r$ where $r^{-1}=\sqrt{c}$. Let us take the usual spherical coordinates $(x,y,z)=(r\sin\phi\,\cos\theta,r\sin\phi\,\sin\theta,r\cos\phi)$ with $0\leq\theta\leq2\pi$ and $0\leq\phi\leq\pi$ the angle with the positive vertical axis. Let $Q$ be the spherical cylinder
\begin{equation}
-r<r\cos\phi_1\leq z\leq r\cos\phi_2<r .
\end{equation}
$\chi(Q)=0$, as it is easy to see. One computes, $\mathrm{area}(Q)=\int_{\phi_2}^{\phi_1}\int_0^{2\pi}r^2\sin\phi\,\dx\theta\,\dx\phi=2\pi r^2(\cos\phi_2-\cos\phi_1)$. One may recall the \textit{horizontal} projection from $S^2_r$ to the cylinder $x^2+y^2=r^2$ is equiareal. Also the two rays $r_1,r_2$ on the boundary components are given by $r_i=r\sin\phi_i$. Recalling \eqref{principalcurvaturesCliffordSpheres} and noticing the outward normal to $Q$ over the boundary, we find $\int_{\partial Q}S_1\,\vol_{\partial Q}=-\frac{\sqrt{r^2-r_2^2}}{rr_2}2\pi  r_2+\frac{\sqrt{r^2-r_1^2}}{rr_1}2\pi r_1=2\pi(\cos\phi_1-\cos\phi_2)$. Hence, as expected, 
\begin{equation} \label{modelexample}
 c\,\mathrm{area}(Q)+\int_{\partial Q}S_1\,\vol_{\partial Q}=0 . 
\end{equation}

Example 2: Let us \textit{otherwise} take a disc cap $Q\subset S^2_s$. Say we 
let $\phi_2=0$ in the same setting above. Then $\chi(Q)=1$ and, since the 
boundary has just one connected component, the identity of Theorem 
\ref{Allendoerfer_odd} is verified, like \eqref{modelexample} but with 
$\frac{\chi(Q)}{2}\omega_{2}=2\pi$ on the right hand side.

Taking the limit in $c\rr0$ in Theorem \ref{Allendoerfer_odd} we 
find the following result (i) again referred by Allendoerfer.
\begin{teo}
 (i) Let $n=2k+1$ be odd and let $Q$ be any relatively-compact domain with 
$C^2$ boundary in Euclidean space of dimension $n+1=2k+2$. Then
 \begin{equation}
  \chi(Q)=\frac{k!}{2\pi^{k+1}}\int_{\partial Q}S_{2k+1}\,\vol_{\partial Q}.
 \end{equation}
(ii) Let $n=2l$ be even and let $M$ be any compact hypersurface of class $C^2$ in Euclidean space of dimension $n+1$. Then
 \begin{equation}
  \chi(M)=\frac{(2l-1)!!}{2^l\pi^{l}}\int_{M}S_{2l}\,\vol_M .
 \end{equation}
\end{teo}
 Case (ii) is stated here for completion; of course, arising from 
\eqref{CaractEuler}. We recall that if $M=\partial Q$ is a boundary, then 
$\chi(\partial Q)=2\chi(Q)$.

\vspace{2mm}
\begin{center}
 \textbf{7. On isoparametric hypersurfaces}
\end{center}
\setcounter{section}{7}

A smooth hypersurface $M$ isometrically immersed in an $(n+1)$-dimensional 
manifold of constant sectional curvature $c$ is isoparametric if and only if 
it has constant principal curvatures. Regarding Euler-Poincar\'e 
characteristic, the formula of Allendoerfer yields some results which might 
add to their further knowledge. We refer to 
\cite{Abresch,Chi,ASiffert,TangXieYan,Thor} for the important sources of 
historical and up to date information on isoparametric hypersurfaces. We 
shall take a short detour by that theory, in so as much as to explain the 
leading ideas. No source we have found has been occupied with $\chi$.

The celebrated work of \'E.~Cartan on isoparametric hypersurfaces starts with the fundamental result that
\begin{equation} \label{Cartanfundamentalformula}
\sum_{j\neq i} m_j\frac{c+\lambda_j\lambda_i}{\lambda_i-\lambda_j}=0 ,
\end{equation}
valid for each distinct principal curvature $\lambda_i$, with multiplicity $m_i$ and $1\leq i\leq g$, where $g$ is the standard for the number of such $\lambda_i$.

If $c\leq0$, then Cartan finds there are at most two such curvatures. If $g=1$, then by \eqref{specialcurvature} and \eqref{characteristicEPspheres} we find ($n=2l$)
\begin{equation}
 \chi(M)=\frac{2(c+\lambda^2)^l}{\omega_n}\vol(M) . 
\end{equation}

If $g=2$, then the fundamental formula yields $c+\lambda_1\lambda_2=0$ and so follows easily from \eqref{GaussCodazzimatrix} that the Euler characteristic
\begin{equation}
 \chi(M)=\begin{cases}
     0 & \mbox{if }m_1\ \mbox{is odd} \\
\frac{4(c+\lambda_1^2)^{m_1/2}(c+\lambda_2^2)^{m_2/2}}{\omega_{m_1}   
\omega_{ m_2}}\vol(M) & \mbox{if }m_1\ \mbox{is even} . \end{cases}
\end{equation}

For $g=1,2$, Cartan gives the classification of the complete hypersurfaces, in \cite{Cartan}: embedded hyperspheres or canonical products of hyperspheres, with $\chi=2$ or $4$, as those from Example in Section 4. For $c<0$, we see the complete hypersurfaces are non compact.

We shall now assume $c=1$, for it is then clear how to recover the general case. Also, by Cartan, each and every $\lambda_i\neq0$.

In case the multiplicities are all equal, denoting $\lambda_1=\cotg\epsilon$ for some $0<\epsilon<\pi$, not a multiple of $\pi/g$, then the distinct principal curvatures are given by
\begin{equation}
 \lambda_i=\cotg\Bigl(\epsilon+(i-1)\frac{\pi}{g}\Bigr),\ \ i=1,\ldots,g .
\end{equation}

For $g=3$, Cartan proves in a second article that the associated multiplicities must all be equal: $m_1=m_2=m_3=m$; moreover they can only assume the values $m=1,2,4,8$. He then determines the associated $3m$-manifolds $M_{(3,m)}$, which were later, in celebrated works of H.F.~M\"unzner, further understood as $m$-sphere tubes of standard Veronese embeddings of projective planes, over the canonical division algebras or the Cayley ring, into ${3m+1}$-dimensional spheres.
\begin{teo}
 For $g=3$ and $m=2,4,8$, we have
 \begin{equation} \label{characEulervolumequandogigualatres}
 \chi(M_{(3,m,\lambda)})=\frac{\calp(\lambda)}{(2\pi)^{\frac{3m}{2}}(1-3\lambda^2)^{m}}\vol(M_{(3,m,\lambda)})
 \end{equation}
 with $\calp(\lambda)\in\Z[\lambda]$ an integer polynomial in $\lambda=\lambda_1$ of degree $3m$. In particular, for $m=2$, we have
 \begin{equation} \label{fantasticincredibleformulaorcoincidence}
  \chi(M_{(3,2,\lambda)})=\frac{3(1+\lambda^2)^3}{\pi^3(1-3\lambda^2)^2}\vol(M_{(3,2,\lambda)}) >0.
 \end{equation}
\end{teo}
\begin{proof}
 Using basic trigonometry identities we deduce
 \[  \lambda_2=\frac{\lambda-\sqrt{3}}{\sqrt{3}\lambda+1},
 \qquad  \lambda_3=\frac{\lambda+\sqrt{3}}{1-\sqrt{3}\lambda} .  \]
 Then we define 
 \[  A=\lambda_2+\lambda_3= \frac{8\lambda}{1-3\lambda^2},\qquad B=\lambda_2\lambda_3=\frac{\lambda^2-3}{1-3\lambda^2} , \]
 and notice every symmetric polynomial $S_{2p}$ in the Pfaffian polynomial from \eqref{CaractEuler} in the $3m$ variables $\lambda,\ldots,\lambda,\lambda_2,\ldots,\lambda_2,\lambda_3,\ldots,\lambda_3$ becomes an integer polynomial in the three variables $\lambda,A,B$. Indeed, we have $S_{2p}(\lambda,\ldots,\lambda_3)=\sum_{j=0}^{m'}\binom{m}{j}\lambda^jS_{2p-j}(\lambda_2,\ldots,\lambda_3)$, where $m'=\min\{m,2p\}$. Now each of those coefficients, indeed symmetric polynomial $S_{2p-j}$, is a polynomial in $A,B$, a result which is easily proved by induction. So the claim follows with the numerator in \eqref{characEulervolumequandogigualatres} in $\Z[\lambda]$. For instance, in case $m=2$, we have the formulas
 \begin{align*}
 & \qquad \qquad  S_0=1,   \qquad   S_2=\lambda_1^2+\lambda_2^2+\lambda_3^2+4(\lambda_1\lambda_2+\lambda_2\lambda_3+\lambda_1\lambda_3)=\lambda^2+4\lambda A+A^2+2B , \\
 &   S_4 =\lambda_1^2\lambda_2^2+\lambda_1^2\lambda_3^2+\lambda_2^2\lambda_3^2+4(\lambda_1^2\lambda_2\lambda_3+\lambda_1\lambda_2^2\lambda_3+\lambda_1\lambda_2\lambda_3^2) =\lambda^2(A^2+2B)+4\lambda AB+B^2
 \end{align*}
 and $S_6=\lambda^2B^2$, which imply the result in \eqref{fantasticincredibleformulaorcoincidence} by not so long computations.
\end{proof}
Of course, the simple form of $\calp$ found for $m=2$ raises the question if a more structural approach should be searched. And if $\calp$ is always a polynomial in $\lambda^2$. Notice we have not had the input from any preferred principal curvature. Due to both sides of \eqref{characEulervolumequandogigualatres}, we must conclude $\calp(\lambda)/(1-3\lambda^2)^m$ is an invariant of the $\lambda_i$. If this is quite clear in case $m=2$, it may help in determining $\calp$ for $m=4,8$.

Perhaps the relevant observation is that the left hand side of \eqref{characEulervolumequandogigualatres} is already known and so we may find a formula for $\vol$ in terms of a single parameter $\lambda$. Indeed, for the division algebras or the octonians $\F=\C,\Hamil,\Octoni$, it is known that $\chi(\F\Proj^2)=3$, so in view of the bundle structure the integer invariant is always $6$.

In case $m=2$, we obtain a unique absolute minimum value, 1, of the function $(1+\lambda^2)^3/(1-3\lambda^2)^2$, at points $0,\pm\sqrt{3}$, which are not admissible $\lambda$ for the construction. It is interesting that
\begin{equation}
 \lim_{\lambda\rightarrow0,\pm\sqrt{3}}
 \vol(M_{(3,2,\lambda)})=2\pi^3 
\end{equation}
is greater than $\omega_7=\pi^4/3$ of the ambient sphere.

We proceed to the following number of distinct principal curvatures. M\"unzner proved that the only possible values for $g$ are $1,2,3,4,6$. And, regarding multiplicities, he deduced the general identity $m_i=m_{i+2}\mod g$, for all $i$. Hence, for $g=4,6$, there are at most two distinct multiplicities, say $m_1$ and $m_2$.

Letting $g=4$, we may compute Euler characteristics with the above technique.
A problem we must leave here is the case of $m_1\neq m_2$ or equal $m_i$ but $>2$.
\begin{teo}  \label{teo_characEulerIsoparametric_g_quatro_m_dois}
	(i) In case $g=4$ and $m_1=m_2=1$, we have $\chi(M)=0$.
	 
\noindent (ii) In case $g=4$ and $m_1=m_2=2$, we have
 \begin{equation}  \label{characEulerIsoparametric_g_quatro_m_dois}
  \chi(M)=\frac{3(\lambda^8-116\lambda^6+316\lambda^4-116\lambda^2+1)}{4\pi^4\lambda^2(1-\lambda^2)^2}\vol(M).
 \end{equation}
 In particular $\chi(M)\neq0$.
\end{teo}
\begin{proof}
As above, we denote $\lambda_1$ by $\lambda$. By simple trigonometry,
\[  \lambda_2= \frac{\lambda-1}{\lambda+1},\qquad\lambda_3= -\frac{1}{\lambda},\qquad\lambda_4= -\frac{1}{\lambda_2}=\frac{\lambda+1}{1-\lambda}  \]
Therefore the symmetric polynomials are simplified, first by the two last helpful identities, and then by substituting $\lambda_2$.
	
When the multiplicities are equal to 1, we find $\chi=3S_0+S_2+3S_4=0$ because $S_0=1$, $S_4=\lambda_1\lambda_2\lambda_3\lambda_4=1$ and further computations yield
\[ S_2= \lambda_1\lambda_2+\lambda_1\lambda_3+ \lambda_1\lambda_4+\lambda_2\lambda_3+\lambda_2\lambda_4+\lambda_3\lambda_4=-6 .  \] 
 If $m_1=m_2=2$, new computations follow:
 \begin{align*}
  S_2&=\lambda_1^2+\lambda_2^2+\frac{1}{\lambda_1^2}+\frac{1}{\lambda_2^2}+4\biggl(\lambda_1\lambda_2-1-\frac{\lambda_1}{\lambda_2}-\frac{\lambda_2}{\lambda_1}-1+\frac{1}{\lambda_1\lambda_2}\biggr)  \\
  &= \frac{1}{\lambda^2(1-\lambda^2)^2}(\lambda^8-24\lambda^6+62\lambda^4-24\lambda^2+1)\\
  S_4&=\lambda_1^2\lambda_2^2+1+\frac{\lambda_1^2}{\lambda_2^2}+\frac{\lambda_2^2}{\lambda_1^2}+1+\frac{1}{\lambda_1^2\lambda_2^2}+4\biggl(\frac{\lambda_2}{\lambda_1}+\frac{\lambda_1}{\lambda_2}-\lambda_2^2-\frac{1}{\lambda_2^2}-\lambda_1\lambda_2 \\
  & \qquad \ \  -\frac{1}{\lambda_1\lambda_2}-\lambda_1\lambda_2-\frac{1}{\lambda_1\lambda_2}-\lambda_1^2-\frac{1}{\lambda_1^2}+\frac{\lambda_1}{\lambda_2}+\frac{\lambda_2}{\lambda_1}\biggr) + 16   \\
  &= \frac{1}{\lambda^2(1-\lambda^2)^2}(-2\lambda^8+62\lambda^6-152\lambda^4+62\lambda^2-2).
 \end{align*}
 Clearly $S_0=S_8=1$, and it is easy to see $S_6=S_2$ because in each summand of $S_6$ there are always two pairs of principal curvatures $\lambda_i,-\inv{\lambda_i}$, which cancel, leaving an intact $S_2$. Finally we recur to \eqref{thedifferentialinvariant}, which yields $\pr=210+90S_2+9S_4$, and recall \eqref{CaractEuler}.
 
 The conclusion that $\chi\neq0$ follows by contradiction. Suppose an 
isoparametric hypersurface is given, with $\lambda_0$ a zero of 
\eqref{characEulerIsoparametric_g_quatro_m_dois}, isolated. We may always 
continuously produce a 1-parameter family of homeomorphic so called parallel 
hypersurfaces, which are still isoparametric of the same type, just by 
varying $\lambda$ in a sufficiently small interval.
\end{proof}

In case (i) of the theorem we have the example of homogeneous isoparametric hypersurface given by the isotropy representation of $\SO(5)/(\SO(3)\times\SO(2))$, cf. \cite[p. 24]{Chi}. Indeed $M$ is a quotient of $\SO(3)\times\SO(2)$ and therefore $\chi(M)=0$.

In case (ii), we may virtually have $\chi$ of any sign. The case considered is classified, cf. \cite{Chi,ASiffert,TangXieYan}, and we know there exists precisely one such hypersurface $M$ up to parallel homotopy, which is homogeneous. We add that it must correspond to one of the `chambers' of the $\lambda_i$ designed by the 8 zeros of the right hand side of \eqref{characEulerIsoparametric_g_quatro_m_dois}.

Since the zeros of \eqref{characEulerIsoparametric_g_quatro_m_dois} are not attained at $0$, $\pm1$ or $\pm\infty$, where $M$ has volume as small as possible, it should be interesting to understand the extreme cases within each bounded end.

As Z.~Tang explained to the author, and indeed consulting \cite{TangXieYan}, the class of the manifold $M$ from Theorem \ref{teo_characEulerIsoparametric_g_quatro_m_dois} (ii) is constructed as a 2-sphere bundle over $\C\Proj^3$. Henceforth we have just
\begin{equation}
  \vol(M)=\frac{32\pi^4\lambda^2(1-\lambda^2)^2}{3(\lambda^8-116\lambda^6+316\lambda^4-116\lambda^2+1)} .
\end{equation}

Finally we consider the case $g=6$. Abresch also discovered that $m_1=m_2$ and that this number $\leq2$. 

Isoparametric hypersurfaces with $g=6$ and multiplicities $m_1=m_2=1$ have vanishing $\chi$. This assertion can be proved by recalling the circle tube structure. Or we may recall the principal curvatures differ by multiples of $\pi/6$ and thus satisfy $\lambda_{i+3}=-\frac{1}{\lambda_i}$ with $i$ mod 6. It is then easy to see $S_4=-S_2$ and $S_6=-S_0=-1$ and hence that $\pr$ must vanish.

The previous result is again consistent with the known homogeneous example induced by $\Gdois/\SO(4)$ via isotropy representation, cf. \cite[p. 24]{Chi}.

\begin{teo}
 In case $g=6$ and $m_1=m_2=2$, we find
 \begin{equation}  \label{characEulerIsoparametric_g_seis_m_dois}
  \chi(M)=\frac{90(1+\lambda^2)^6}{\pi^6\lambda^2(1-3\lambda^2)^2(3-\lambda^2)^2}\vol(M) >0.
 \end{equation}
\end{teo}
\begin{proof}
 Let $\lambda_1=\lambda$ as usual. Due to the step increase by $\pi/6$, we have
 \[  \lambda_2=\frac{\sqrt{3}\lambda-1}{\lambda+\sqrt{3}},\qquad\lambda_3=\frac{\lambda-\sqrt{3}}{\sqrt{3}\lambda+1}, \qquad \lambda_4\lambda_1=\lambda_5\lambda_2=\lambda_6\lambda_3=-1 .  \]
 In order to deal with $\lambda_2$ and $\lambda_3$ we define
 \[ s=\lambda^2-3,\qquad t=3\lambda^2-1,\qquad  p=\sqrt{3}\lambda^2-4\lambda+\sqrt{3},\qquad \overline{p}=\sqrt{3}\lambda^2+4\lambda+\sqrt{3}  \]
 so that a number of most helpful identities occur:
 \[ \lambda_2=\frac{p}{s},\qquad\lambda_3=\frac{p}{t},\qquad p\overline{p}=st=3\lambda^4-10\lambda^2+3 ,\qquad  p-\overline{p}=-8\lambda .  \]
 These enable a more simple treatment of the symmetric polynomials. Indeed, up to a common denominator $\lambda^2s^2t^2$, we may rewrite $S_{2j}$ as polynomials in $\lambda, p\overline{p},p-\overline{p},s+t,s^2+t^2,s^3+t^3,s^4+t^4$. The result is:
 \begin{align*}
   S_2 &=\frac{1}{\lambda^2s^2t^2}\bigl(9(\lambda^{12} + 1)) - 540(\lambda^{10} + \lambda^2) + 4095(\lambda^8 + \lambda^4) - 7608\lambda^6\bigr) \\
     S_4 &=\frac{1}{\lambda^2s^2t^2}\bigl(-60(\lambda^{12} + 1)+4095(\lambda^{10} + \lambda^2) - 30600(\lambda^8 + \lambda^4) +57210\lambda^6\bigr) \\
  S_6 &=\frac{2}{\lambda^2s^2t^2}\bigl(59(\lambda^{12}+1) - 3804(\lambda^{10}+\lambda^2) + 28605(\lambda^8+\lambda^4) - 53336 \lambda^6 \bigr).
   \end{align*} 
 With $S_8$ we notice that 2 pairs of principal curvatures with indices $i$ 
and $i+3$ mod 6 must lie within each summand. Then a simple cancellation 
yields an $S_4$. The same happens with 4 and 6 pairs for $S_{10}=S_2$ and 
$S_{12}=S_0=1$, respectively.  Therefore we have
 \begin{align*}
  \pr &= 11!!(S_0+S_{12})+9!!(S_2+S_{10})+7!!3!!(S_4+S_8)+5!!5!!S_6 \\
  &= 10395\cdot2+945\cdot2S_2+315\cdot2S_4+225S_6 \\
  &=\frac{2}{\lambda^2s^2t^2}\biggl(10395(9(\lambda^{10} + \lambda^2) - 60(\lambda^8 + \lambda^4) + 118\lambda^6)+ \\
  & \qquad + (945\cdot9-315\cdot60+225\cdot59)(\lambda^{12}+1) \\
  & \qquad + (-945\cdot540+315\cdot4095-225\cdot3804)(\lambda^{10}+\lambda^2)\\ & \qquad +(945\cdot4095-315\cdot30600+225\cdot28605)(\lambda^{8}+\lambda^4) \\ & \qquad +(-945\cdot7608+315\cdot57210-225\cdot53336)\lambda^6\biggr)  \\
  & = \frac{2}{\lambda^2s^2t^2}\cdot2880(1 + \lambda^2)^6  .
 \end{align*}
 Of course we have recurred to a computer program for these very last computations and simplification.
\end{proof}
 Due to the surprising reduction of the polynomial, as we had seen before with \eqref{fantasticincredibleformulaorcoincidence}, a quest is clearly set to understand more deeply the theorem above.

\vspace{7mm}

The author acknowledges the fruitful conversations with Jos\'e Navarro, from 
the U.~Extremadura, in Portugal and Spain, with Zizhou Tang, U. Nankai, at CIM, and with Miguel Dom\'\i nguez V\'azquez, U. Santiago de Compostela, which helped to improve some of the conclusions above. He also thanks deeply an anonymous Referee.

Parts of this article were written while the author was a visitor at Chern Institute of Mathematics, Tianjin, China. He warmly thanks CIM for the excellent conditions provided and for the opportunity to visit such an inspiring campus of the University of Nankai.

The author dedicates this work to the commemorations of the 500th anniversary of the first global circumnavigation voyage by the Portuguese and Spanish navigators.

\vspace{14mm}

\textsc{R. Albuquerque}\ \ \ \textbar\ \ \ 
{\texttt{rpa@uevora.pt}}

Centro de Investiga\c c\~ao em Mate\-m\'a\-ti\-ca e Aplica\c c\~oes

Rua Rom\~ao Ramalho, 59, 671-7000 \'Evora, Portugal

Research leading to these results has received funding from Funda\c 
c\~ao para a Ci\^encia e a Tecnologia (UID/MAT/04674/2013).

\end{document}